\documentclass[reqno]{amsart}
\usepackage{amssymb}
\usepackage{hyperref}

\newtheorem{theorem}{Theorem}[section]
\newtheorem{lemma}[theorem]{Lemma}
\newtheorem{proposition}[theorem]{Proposition}
\newtheorem{corollary}[theorem]{Corollary}
\theoremstyle{definition}
\newtheorem{definition}[theorem]{Definition}

\theoremstyle{remark}
\newtheorem{remark}[theorem]{Remark}

\numberwithin{equation}{section}

\begin{document}

\title{On Supercyclic Sets of Operators}
\author{ Mohamed Amouch and Otmane Benchiheb }

\address{\textsc{Mohamed Amouch and Otmane Benchiheb},
University Chouaib Doukkali.
Department of Mathematics, Faculty of science
Eljadida, Morocco}
\email{amouch.m@ucd.ac.ma}
\email{otmane.benchiheb@gmail.com}

\begin{abstract}
Let $X$ be a complex topological vector space with dim$(X)>1$ and $\mathcal{B}(X)$ the space of all continuous linear operators on $X$. In this paper, we extend the concept of supercyclicity of a single operators and strongly continuous semigroups of operators to a subset of $\mathcal{B}(X)$. We establish some results for supercyclic set of operators and we give some applications for strongly continuous semigroups of operators and $C$-regularized group of operators.
\end{abstract}


\maketitle

\section{Introduction and Preliminary}
Let $X$ be a complex topological vector space with dim$(X)>1$ and $\mathcal{B}(X)$ the space of all continuous linear operators on $X$. By an operator, we always mean a continuous linear operator.
If $T\in\mathcal{B}(X)$, then the orbit of a vector $x\in X$ under $T$ is the set
$$ Orb(T,x):=\{T^{n}x\mbox{ : }n\geq0\}. $$
The operator $T$ is said to be supercyclic \cite{HW} if there is some vector $x\in X$ such that 
$$\mathbb{C}Orb(T,x)=\{\alpha T^{n}x\mbox{ : }\alpha\in\mathbb{C}\mbox{, } n\geq0\}$$
 is dense in $X$. Such vector $x$ is called a supercyclic vector for $T$, and the set of all supercyclic vectors for $T$ is denoted by $SC(T)$.
Recall \cite{Feldman1} that $T$ is supercyclic if and only if for each pair $(U,V)$ of nonempty open subsets of $X$ there exist $\alpha\in\mathbb{C}$ and $n\geq0$ such that 
$$\alpha T^{n}(U)\cap V\neq \emptyset.$$

The supercyclicity criterion for a single operator was introduced in \cite{Salas2}. It provides several sufficient conditions that ensure supercyclicity.
We say that an operator $T\in\mathcal{B}(X)$ satisfies the supercyclicity criterion  if there exist an increasing sequence of integers $(n_k)$, a sequence $(\alpha_{n_k})$ of nonzero complex numbers, two dense sets $X_0$, $Y_0\subset X$ and a sequence of maps $S_{n_k}$ : $Y_0\longrightarrow X$ such that$:$
\begin{itemize}
\item[$(i)$] $\alpha_{n_k}T^{n_k}x\longrightarrow 0$ for any $x\in X_0$;
\item[$(ii)$] $\alpha_{n_k}^{-1}S_{n_k}y\longrightarrow 0$ for any $y\in Y_0$;
\item[$(iii)$] $T^{n_k} S_{n_k}y\longrightarrow y$ for any $y\in Y_0$.
\end{itemize}
For a general overview of supercyclicity and related proprieties in linear dynamics see \cite{Bayart Matheron,BBP,GEU,Erdmann Peris,HL,HW,LSM,MS,Salas2}.

Recall \cite[Definition 1.5.]{Erdmann Peris}, that an operator $T\in\mathcal{B}(X)$ is called quasi-conjugate or quasi-similar to an operator $S\in\mathcal{B}(Y)$ if there exists a continuous map $\phi$ : $Y\longrightarrow X$ with dense range such that $T\circ\phi=\phi\circ S.$ If $\phi$ can be chosen to be a homeomorphism, then $T$ and $S$ are called conjugate or similar.
Recall \cite[Definition 1.7.]{Erdmann Peris}, that a property $\mathcal{P}$ is said to be preserved under quasi-similarity if the following holds$:$ if an operator $T\in\mathcal{B}(X)$ has property $\mathcal{P}$, then every operator $S\in\mathcal{B}(Y)$ that is quasi-similar to $T$ has also property $\mathcal{P}$. 

Recall \cite{AKH}, that  If $\Gamma$ is a subset of $\mathcal{B}(X),$ then the orbit of a vector $x\in X$ under $\Gamma$ is the set
$$ Orb(\Gamma,x)=\{Tx \mbox{ : }T\in\Gamma\}. $$

Recall \cite{Conway}, that for a topological vector space $X$, the strong operator topology (SOT) on $\mathcal{B}(X)$ is the topology with respect to which any $T\in \mathcal{B}(X)$ has a neighborhood basis consisting of sets of the form
$\Omega=\{S\in\mathcal{B}(X) \mbox{ : }Se_i-Te_i\in U\mbox{, }i=1,2,\dots,k\}$ where $k\in\mathbb{N}$, $e_1,e_2,\dots e_k\in X$ are linearly independent
and U is a neighborhood of zero in $X$.

In this work, we introduce the concept of supercyclicity for a set of operators of $\mathcal{B}(X)$ which
generalize the notion of supercyclicity for a single operator. We deal with supercyclic set and we prove that some properties known for one supercyclic operator remain true for supercyclic set of operators.
In \cite{Bayart Matheron}, it has shown that the set of supercyclic vectors of a single operator is a $G_\delta$ set. In section 2, we show that this result holds for the set of supercyclic vectors of a set of operators and we prove that supercyclicity is preserved under quasi-similarity.
 In section 3, we introduce notions of supercylic transitive sets, strictly transitive sets, supertransitive sets, and the notion of supercyclic criterion for sets of operators. We give relations between these notions and the concept of supercyclic sets of operators and we prove that these notions are preserved under quasi-similarity or similarity. In section 4, we give some results for $C$-regularized group of operators. We prove that, if $(S(z))_{z\in\mathbb{C}}$ is a supercyclic $C$-regularized group of operators and $C$ has dense range, then $(S(z))_{z\in\mathbb{C}}$ is supercyclic transitive. At the end, we give some applications for strongly continuous semigroups of operators.
\section{ Supercyclic Sets of Operators}
\begin{definition}
Let $\Gamma\subset\mathcal{B}(X)$. We say that $\Gamma$ is a supercyclic set of operators or a supercyclic set if there exists $x\in X$ such that 
$$\mathbb{C} Orb(\Gamma,x):=\{\alpha Tx \mbox{ : }\alpha\in\mathbb{C}\mbox{, }T\in\Gamma\}$$
is a dense subset of $X$. 
Such vector $x$ is called a supercyclic vector for $\Gamma$ or a supercyclic vector. The set of all supercyclic vectors for $\Gamma$ is denoted by $SC(\Gamma)$.
\end{definition}
\begin{remark}
Let $T\in\mathcal{B}(X)$. Then $T$ is supercyclic as an operator if and only if the set
$$\Gamma=\{T^n\mbox{ : }n\geq0\}$$
is supercyclic as a set of operators.
\end{remark}
Let $\Gamma\subset\mathcal{B}(X).$ We denote by $\{\Gamma\}^{'}$ the set of all elements of $\mathcal{B}(X)$ which commute with every element of $\Gamma.$
\begin{proposition}\label{p1}
Let $\Gamma\subset\mathcal{B}(X)$ be a supercyclic set and $T$ an operator with dense. If $T\in\{\Gamma\}^{'}$, then $Tx\in SC(\Gamma),$ for all $x\in SC(\Gamma).$
\end{proposition}
\begin{proof}
Let $O$ be a nonempty open subset of $X$. Since $T$ is continuous and of a dense range, $T^{-1}(O)$ is a nonempty open subset of $X$. Let $x\in SC(\Gamma)$, then there exist $\alpha\in\mathbb{C}$ and $S\in \Gamma$ such that $\alpha Sx\in T^{-1}(O)$, that is $\alpha T(Sx)\in O$. Since $T\in \{\Gamma\}^{'}$, it follows that
$$\alpha S(Tx)=\alpha T(S x)\in O.$$
Hence, $\mathbb{C}Orb(\Gamma,Tx)$  meets every nonempty open subset of $X$. From this, we deduce that $\mathbb{C}Orb(\Gamma,Tx)$ is dense in $X$. That is, $Tx\in SC(\Gamma)$. 
\end{proof}
\begin{corollary}
Let $\Gamma\subset\mathcal{B}(X)$ be a supercyclic set. If $x\in SC(\Gamma)$, then $\alpha x\in SC(\Gamma)$, for all $\alpha \in \mathbb{C}\setminus\{0\}.$
\end{corollary}
\begin{proof}
Let $\alpha \in \mathbb{C}\setminus\{0\}$ and $x\in SC(\Gamma)$. Then $T=\alpha I$ has dense range and $T\in \{\Gamma\}^{'}$. By Proposition \ref{p1}, we deduce that $\alpha x\in SC(\Gamma)$.
\end{proof}

Let $X$ and $Y$ be topological vector spaces and let $\Gamma\subset \mathcal{B}(X)$ and $\Gamma_1\subset \mathcal{B}(Y)$. Recall \cite{AO}, that $\Gamma$ and $\Gamma_1$ are called quasi-similar if there exists a continuous map $\phi$ : $X\longrightarrow Y$ with dense range such that for all $T\in\Gamma,$ there exists $S\in\Gamma_1$ satisfying $S\circ\phi=\phi\circ T$. If $\phi$ is a
homeomorphism, then $\Gamma$ and $\Gamma_1$ are called similar.


It has shown in \cite{Erdmann Peris} that the supercyclicity of a single operator is preserved under quasi-similarity.
In the following, we prove that the same result holds for sets of operators.
\begin{proposition}\label{14}
Let $X$ and $Y$ be topological vector spaces and let $\Gamma\subset \mathcal{B}(X)$ and $\Gamma_1\subset \mathcal{B}(Y)$. If $\Gamma$ and $\Gamma_1$ are quasi-similar, then $\Gamma$ is supercyclic in $X$ implies that $\Gamma_1$ is supercyclic in $Y$. Moreover, 
$$\phi(SC(\Gamma)\subset SC(\Gamma_1).$$
\end{proposition}
\begin{proof}
Let $O$ be a nonempty open subset of $Y$, then $\phi^{-1}(O)$ is a nonempty open subset of $X$. If $x\in SC(\Gamma)$, then there exist $\alpha\in\mathbb{C}$ and $T\in\Gamma$ such that $\alpha Tx\in \phi^{-1}(O)$, that is $\alpha\phi(Tx)\in O$. Let $S\in\Gamma_1$ such that $S\circ\phi=\phi\circ T$. Hence,
$$\alpha S(\phi x) =\alpha\phi(Tx)\in O.$$
Thus, $\mathbb{C}Orb(\Gamma_1,\phi x)$ meets every nonempty open subset of $Y$. From this, we deduce that $\mathbb{C}Orb(\Gamma_1,\phi x)$ is dense in $Y$. That is, $\Gamma_1$ is supercyclic and $\phi x\in SC(\Gamma_1)$.
\end{proof}
\begin{corollary}
Let $X$ and $Y$ be topological vector spaces and let $\Gamma\subset \mathcal{B}(X)$ and $\Gamma_1\subset \mathcal{B}(Y)$. If $\Gamma$ and $\Gamma_1$ are similar, then $\Gamma$ is supercyclic in $X$ if and only if $\Gamma_1$ is supercyclic in $Y$. Moreover, 
$$\phi(SC(\Gamma)= SC(\Gamma_1).$$
\end{corollary}
\begin{proposition}\label{p3}
Let $\Gamma\subset\mathcal{B}(X)$ and $\{\alpha_T\}_{T\in\Gamma}$ a sequence of nonzero complex numbers. Then, $\Gamma$ is supercyclic if and only if $\Gamma_1=\{ \alpha_T T\mbox{ : }T\in\Gamma \}$ is supercyclic. Moreover, $\Gamma$ and $\Gamma_1$ have the same supercyclic vectors.
\end{proposition}
\begin{proof}
If $x\in X$, then $\mathbb{C}Orb(\Gamma,x)=\mathbb{C}Orb(\Gamma_1,x).$ Hence, $\overline{\mathbb{C}Orb(\Gamma,x)}=\overline{\mathbb{C}Orb(\Gamma_1,x)}$ and the proof is completed.
\end{proof}
Let $\{X_i\}_{i=1}^{n}$ be a family of topological vector spaces and let $\Gamma_i$ be a subset of $\mathcal{B}(X_i)$, for all $1\leq i\leq n$. Define 
$$\oplus_{i=1}^nX_i=X_1\times X_2\times\dots \times X_n=\{(x_1,x_2,\dots,x_n) \mbox{ : }x_i\in X_i\mbox{, }1\leq i\leq n\},$$
and
$$\oplus_{i=1}^n\Gamma_i=\{T_1\times T_2\times\dots \times T_n\mbox{ : }T_i\in\Gamma_i\mbox{, }1\leq i\leq n\}.$$
\begin{proposition}\label{p4}
Let $\{X_i\}_{i=1}^{n}$ be a family of topological vector spaces and $\Gamma_i$ a subset of
$\mathcal{B}(X_i),$ for all $1\leq i\leq n$. If $\oplus_{i=1}^n\Gamma_i$ is a supercyclic set in $\oplus_{i=1}^n X_i$, then $\Gamma_i$ is a supercyclic set in $X_i$, for all $1\leq i\leq n$. Moreover, if $(x_1,x_2,\dots,x_n)\in SC(\oplus_{i=1}^n\Gamma_i)$, then $x_i\in SC(\Gamma_i)$, for all $1\leq i\leq n$.
\end{proposition}
\begin{proof}
Let $(x_1,x_2,\dots,x_n)\in SC(\oplus_{i=1}^n\Gamma_i)$. For all $1\leq i\leq n$, let $O_i$ be a nonempty open subset of $X_i$, then $ O_1\times O_2\times\dots\times O_n$ is a nonempty open subset of $\oplus_{i=1}^n X_i$. Since $Orb(\oplus_{i=1}^n\Gamma_i,\oplus_{i=1}^n x_i)$ is dense in $\oplus_{i=1}^n X_i$, there exist $\alpha\in\mathbb{C}$ and $T_i\in\Gamma_i$; $1\leq i\leq n$ such that
$$(\alpha T_1 x_1,\alpha T_2 x_2,\dots,\alpha T_n x_n)=\alpha (T_1\times T_2\times\dots \times T_n)(x_1,x_2,\dots,x_n) \in  O_1\times O_2\times\dots\times O_n,$$
that is $\alpha T_i x_i\in O_i$, for all $1\leq i\leq n$. Hence, $\Gamma_i$ is a supercyclic set in $X_i$ and $x_i\in SC(\Gamma_i)$, for all $1\leq i\leq n$. 
\end{proof}

A subset of $X$ is said to be $G_\delta$ type if it's an intersection of a countable collection of open sets, see \cite{Willered}.
\begin{proposition}\label{p5}
Let $X$ be a second countable topological vector space and $\Gamma\subset\mathcal{B}(X)$ a supercyclic set. Then,
$$SC(\Gamma)=\bigcap_{n\geq1}\left(\bigcup_{\beta\in\mathbb{C}\setminus\{0\}}\bigcup_{T\in \Gamma} T^{-1}(\beta U_{n})\right),$$
where $(U_n)_{n\geq1}$ is a countable basis of the topology of $X$. As a consequence, $SC(\Gamma)$ is a $G_\delta$ type set.
\end{proposition}
\begin{proof}
Suppose that $\Gamma$ is a supercyclic set. Then, $x\in SC(\Gamma)$ if and only if $\overline{\mathbb{C}Orb(\Gamma,x)} = X$.
Equivalently, for all $n\geq1$ we have $U_n\cap \mathbb{C}Orb(\Gamma,x)\neq\emptyset$. That is, for all $n \geq 1$ there exist $\alpha\in\mathbb{C}$ and
$T\in\Gamma$ such that $x \in \alpha T^{-1}(U_n)$. This is equivalent to the fact that 
$\displaystyle x\in\bigcap_{n\geq1}\left(\bigcup_{\beta\in\mathbb{C}\setminus\{0\}}\bigcup_{T\in \Gamma} T^{-1}(\beta U_{n})\right)$. Hence, $\displaystyle SC(\Gamma)=\bigcap_{n\geq1}\left(\bigcup_{\beta\in\mathbb{C}\setminus\{0\}}\bigcup_{T\in \Gamma} T^{-1}(\beta U_{n})\right)$.

Since $\displaystyle \bigcup_{\beta\in\mathbb{C}\setminus\{0\}}\bigcup_{T\in \Gamma} T^{-1}(\beta U_{n})$ is an open subset of X, for all $n \geq 1$, $SC(\Gamma)$ is a $G_\delta$ type.                                                   
\end{proof}
\section{Density and Transitivity of Sets of Operators} 
Supercyclic transitivity of a single operator was introduced in \cite{Feldman1}. In the following definition, we extend this notion to sets of operators.
\begin{definition}
Let $\Gamma\subset\mathcal{B}(X)$. We say that $\Gamma$ is a supercyclic transitive set of operators or a supercyclic tarnsitive set, if for each pair of nonempty open subsets $(U,V)$ of $X$, there exists some $\alpha\in\mathbb{C}\setminus\{0\}$ and some $T\in \Gamma$ such that 
$$T(\alpha U)\cap V\neq\emptyset.$$
\end{definition}
\begin{remark}
Let $T\in\mathcal{B}(X)$. Then $T$ is supercyclic transitive as an operator if and only if the set
$$\Gamma=\{T^n\mbox{ : }n\geq0\}$$
is supercyclic transitive as a set of operators.
\end{remark}

The supercyclic transitivity of a single operators is preserved under quasi-similarity, see \cite[Proposition 1.13]{Erdmann Peris}. The following proposition proves that the same result holds for sets of operators.
\begin{proposition}\label{prop1}
Let $X$ and $Y$ be topological vector spaces and let $\Gamma\subset \mathcal{B}(X)$ be quasi-similar to $\Gamma_1\subset \mathcal{B}(Y)$. If $\Gamma$ is supercyclic transitive in $Y$, then $\Gamma_1$ is supercyclic transitive in $Y$.
\end{proposition}
\begin{proof}
Since $\Gamma$ and $\Gamma_1$ are quasi-similar, there exists a continuous map $\phi$ : $X\longrightarrow Y$ with dense range such that for all $T\in\Gamma,$ there exists $S\in\Gamma_1$ satisfying $S\circ\phi=\phi\circ T$. 
Let $U$ and $V$ be nonempty open subsets of $X$. Since $\phi$ is continuous and of dense range, $\phi^{-1}(U)$ and $\phi^{-1}(V)$ are nonempty and open. Since $\Gamma$ is supercyclic transitive in $X$, there exist  $y\in \phi^{-1}(U)$ and $\alpha\in\mathbb{C},$ $T\in\Gamma$ with $\alpha Ty\in\phi^{-1}(V)$, which implies that $\phi(y)\in U$ and $\alpha \phi(Ty)\in V$. Let $S\in\Gamma$ such that $S\circ\phi=\phi\circ T$. Then, $\phi(y)\in U$ and $\alpha S\phi(y)\in V$. Thus, $\alpha S(U)\cap V\neq\emptyset.$ Hence, $\Gamma_1$ is supercyclic transitive in $Y.$
\end{proof}
\begin{corollary}
Let $X$ and $Y$ be topological vector spaces and let $\Gamma\subset \mathcal{B}(X)$ be similar to $\Gamma_1\subset \mathcal{B}(Y)$. Then, $\Gamma$ is supercyclic transitive in $Y$ if and only if $\Gamma_1$ is supercyclic transitive in $Y$.
\end{corollary}

In the following result, we give necessary and sufficient conditions for a set of operators to be supercyclic transitive.
\begin{theorem}\label{tt}
Let $X$ be a normed space and $\Gamma\subset \mathcal{B}(X)$. The following assertions are equivalent: 
\begin{itemize}
\item[$(i)$] $\Gamma$ is supercyclic transitive;
\item[$(ii)$] For each $x$, $y\in X,$ there exists sequences $\{k\}$ in $\mathbb{N}$, $\{x_k\}$ in $X$, $\{\alpha_k\}$ in $\mathbb{C}$ and $\{T_k\}$ in $\Gamma$ such that
$$x_k\longrightarrow x\hspace{0.3cm}\mbox{ and }\hspace{0.3cm}T_k(\alpha_k x_k)\longrightarrow y;$$
\item[$(iii)$] For each $x$, $y\in X$ and for $W$ a neighborhood of $0$, there exist $z\in X$, $\alpha\in\mathbb{C}$ and $T\in\Gamma$  such that 
$$x-z\in W \hspace{0.3cm}\mbox{ and }\hspace{0.3cm} T(\alpha z)-y\in W. $$ 
\end{itemize}
\end{theorem}
\begin{proof}$(i)\Rightarrow(ii)$
Let $x$, $y\in X$. For all $k\geq1$, let $U_k=B(x,\frac{1}{k})$ and $V_k=B(y,\frac{1}{k})$. Then $U_k$ and $V_k$ are nonempty open subsets of $X$. Since $\Gamma$ is supercyclic transitive, there exist $\alpha_k\in\mathbb{C}$ and $T_k\in\Gamma$ such that $T_k(\alpha_k U_k)\cap V_k\neq\emptyset$. For all $k\geq1$, let $x_k\in U_k$ such that $ T_k(\alpha x_k)\in V_k$, then 
$$\Vert x_k-x \Vert<\frac{1}{k}\hspace{0.3cm}\mbox{ and }\hspace{0.3cm}\Vert T_k(\alpha x_k)-y \Vert<\frac{1}{k}$$
which implies that 
$$x_k\longrightarrow x\hspace{0.3cm}\mbox{ and }\hspace{0.3cm} T_k(\alpha x_k)\longrightarrow y.$$

$(ii)\Rightarrow(iii)$ Let $x$, $y\in X$. There exists sequences $\{k\}$ in $\mathbb{N}$, $\{x_k\}$ in $X$  $\{\alpha_k\}$ in $\mathbb{C}$ and $\{T_k\}$ in $\Gamma$ such that
$$x_k-x\longrightarrow 0\hspace{0.3cm}\mbox{ and }\hspace{0.3cm}T_k(\alpha_k x_k)-y\longrightarrow 0.$$
Let $W$ be a neighborhood of zero, then there exists $N\in\mathbb{N}$ such that 
$$x-x_k\in W\hspace{0.3cm}\mbox{ and }\hspace{0.3cm}T_k(\alpha_k x_k)-y\in W,$$ 
for all $k\geq N$.

$(iii)\Rightarrow(i)$ Let $U$ and $V$ be two nonempty open subsets of $X$. There exists $x$, $y\in X$ such that $x\in U$ and $y\in V$. Since for all $k\geq1$,  $W_k=B(0,\frac{1}{k})$ is a neighborhood of $0$, there exist $z_k\in X,$ $\alpha_k\in\mathbb{C}$ and $T_k\in\Gamma$ such that 
$$\Vert x-z_k\Vert<\frac{1}{k}\hspace{0.3cm}\mbox{ and }\hspace{0.3cm}\Vert T_k(\alpha_k z_k)-y\Vert<\frac{1}{k}.$$
This implies that 
$$z_k\longrightarrow x\hspace{0.3cm}\mbox{ and }\hspace{0.3cm}T_k(\alpha_k z_k)\longrightarrow y.$$
Since $U$ and $V$ are nonempty open subsets of $X$, $x\in U$ and $y\in V$, there exists $N\in\mathbb{N}$ such that $z_k\in U$ and $T_k(\alpha_k z_k)\in V$, for all $k\geq N.$ 
\end{proof}

An operator is supercyclic if and only if it is supercyclic transitive\cite{Feldman1}. Let $\Gamma\subset\mathcal{B}(X)$. In what follows, we prove that $\Gamma$ is topologically transitive implies that $\Gamma$ is supercyclic.
\begin{theorem}\label{t1}
Let $X$ be a second countable Baire topological vector space and $\Gamma$ a subset of $\mathcal{B}(X)$. The following assertions are equivalent$:$
\begin{itemize}
\item[$(i)$] $SC(\Gamma)$ is dense in $X$;
\item[$(ii)$] $\Gamma$ is supercylic transitive.
\end{itemize}
As a consequence$;$ a supercyclic transitive set is supercyclic.
\end{theorem}
\begin{proof} Since $X$ is a second countable topological vector space$,$ we can consider $(U_m)_{m\geq1}$ a countable basis of the topology of $X.$\\  
$(i)\Rightarrow (ii) :$ Assume that $SC(\Gamma)$ is dense in $X$ and let $U$, $V$ be two nonempty open subsets of $X$.  By Proposition \ref{p5}, we have 
$$SC(\Gamma)=\bigcap_{n\geq1}\left(\bigcup_{\beta\in\mathbb{C}\setminus\{0\}}\bigcup_{T\in \Gamma} T^{-1}(\beta U_{n})\right).$$
Hence, for all $n\geq 1$ the set $\displaystyle A_n=\bigcup_{\beta\in\mathbb{C}\setminus\{0\}}\bigcup_{T\in \Gamma} T^{-1}(\beta U_{n})$ is dense in $X$. Thus, for all $n,$ $m\geq 1$, we have $A_n\cap U_m\neq\emptyset$ which implies that for all $n,$ $m\geq 1$ there exist $\beta\in\mathbb{C}\setminus\{0\}$ and $T\in \Gamma$, such that $T(\beta U_m)\cap U_n\neq\emptyset$. Since $(U_m)_{m\geq1}$ is a countable basis of the topology of $X$, it follows that $\Gamma$ is supercyclic transitive.\\
$(ii)\Rightarrow (i) : $ Let $n,$ $m\geq 1$, then there exist $\beta\in \mathbb{C}\setminus\{0\}$ and $T\in \Gamma$ such that $T(\beta U_m)\cap U_n\neq \emptyset$ which implies that  $T^{-1}(\frac{1}{\beta} U_n)\cap U_m\neq \emptyset$. Hence, for all $n\geq1$ the set $\displaystyle\bigcup_{T\in \Gamma}\bigcup_{\beta\in\mathbb{C}\setminus\{0\}} T^{-1}(\beta U_{n})$ is dense in $X$. Since $X$ is a Baire space, it follows that 
$$SC(\Gamma)=\bigcap_{n\geq1}\left(\bigcup_{\beta\in\mathbb{C}\setminus\{0\}}\bigcup_{T\in \Gamma} T^{-1}(\beta U_{n})\right)$$
is a dense subset of $X$.
\end{proof}
The converse of Theorem \ref{t1} is holds with additional assumptions, to show that, we need the next lemma.
\begin{lemma}\label{le1}
Let $X$ be a topological vector space without isolated point and let $\Gamma\subset\mathcal{B}(X)$ a supercyclic set. If $T\in\Gamma$, then $\Gamma\setminus\{T\}$ is supercyclic. Moreover, $\Gamma$ and $\Gamma\setminus\{T\}$ have the same supercyclic vectors.
\end{lemma}
\begin{proof}
Let $O$ be a nonempty open subset of $X$ and $T\in\Gamma.$ If $x\in SC(\Gamma)$, then $O\setminus $span$\{Tx\}$ is a nonempty open subset of $X.$ There exist $\alpha\in\mathbb{C}$ and $S\in \Gamma\setminus\{T\}$
such that 
$$\alpha Sx\in O\setminus \mbox{span}\{Tx\}\subset O.$$
Hence, $\mathbb{C} Orb(\Gamma\setminus\{T\},x)$  meets every nonempty open subset of $X$. From this, we deduce that $\mathbb{C}Orb(\Gamma\setminus\{T\},x)$ is a dense subset of $X$. That is, 
$\Gamma\setminus\{T\}$ is a supercyclic set and 
$$SC(\Gamma)=SC(\Gamma\setminus\{T\}).$$
\end{proof}
\begin{theorem}\label{so}
Assume that $X$ is without isolated point and let $\Gamma\subset\mathcal{B}(X)$ such that for all $T$, $S\in\Gamma$ with $T\neq S$, there exists $A\in\Gamma$ such that $T=AS$. Then $\Gamma$ is supercyclic implies that $\Gamma$ is supercyclic transitive.
\end{theorem}
\begin{proof}
Since $\Gamma$ is supercyclic, there exists $x\in X$ such that $\mathbb{C}Orb(\Gamma,x)$ is a dense subset of $X$. By Lemma \ref{le1} and since $X$ is without isolated point, we may suppose that $I\in\Gamma$ and $\mathbb{C}\setminus\{0\}Orb(\Gamma,x)$ is dense in $X$.
Let $U$ and $V$ be two nonempty open subsets of $X$, then there exist $\alpha$, $\beta\in\mathbb{C}\setminus\{0\}$, and $T$, $S\in\Gamma$ such that
\begin{equation}\label{e1}
\alpha Tx\in U \hspace{0.6cm}\mbox{ and }\hspace{0.6cm}\beta Sx\in V. 
\end{equation}
If $T=S$, then $U\cap V\neq \emptyset$ which means that $I(U)\cap V\neq\emptyset.$ Since $I\in\Gamma$, the result holds.\\
If $T\neq S$, then there exists $A\in\Gamma$ such that $T=AS$. By (\ref{e1}), we have
$$\alpha A(Sx)\in U \hspace{0.2cm}\mbox{ and }\hspace{0.2cm}\beta A(Sx)\in  A(V)$$
which implies that $U\cap A(\frac{\alpha}{\beta} V)\neq \emptyset$. 
Hence, $\Gamma$ is supercyclic transitive.
\end{proof}
\begin{definition}
A set $\Gamma\subset\mathcal{B}(X)$ is called strictly transitive if for each pair of nonzero elements $x,$ $y$ in $X$, there exist some $\alpha\in\mathbb{C}$ and $T\in\Gamma$ such that $\alpha Tx=y.$ An operator $T\in \mathcal{B}(X)$ is strictly transitive if 
$$\Gamma=\{T^n\mbox{ : }n\geq0\} $$
is a strictly transitive set.
\end{definition}
\begin{proposition}
A strictly transitive set is supercyclic transitive. As a consequence, a strictly transitive set is supercyclic.
\end{proposition}
\begin{proof}
Let $\Gamma\subset\mathcal{B}(X)$ be a strictly transitive set. If $U$ and $V$ are two nonempty open subsets of $X$, there exist $x,$  $y\in X$ such that
$x\in U$ and $y\in V$. Since $\Gamma$ is strictly transitive, there exist $\alpha\in\mathbb{C}$ and $T\in\Gamma$ such that $\alpha Tx=y.$ Hence,
$$ \alpha Tx\in \alpha T(U)\hspace{0.3cm} \mbox{ and }\hspace{0.3cm}\alpha Tx\in V. $$
Thus, $\alpha T(U)\cap V\neq\emptyset,$ which implies that $\Gamma$ is supercyclic transitive.
\end{proof}

In the following proposition, we prove that strictly transitivity of sets of operators is preserved under similarity.
\begin{proposition}
Let $X$ and $Y$ be topological vector spaces and let $\Gamma\subset \mathcal{B}(X)$ be similar to $\Gamma_1\subset \mathcal{B}(Y)$. Then $\Gamma$ is strictly transitive in $X$ if and only $\Gamma_1$ is strictly transitive in $Y.$
\end{proposition}
\begin{proof}
Since $\Gamma$ and $\Gamma_1$ are similar, there exists a homeomorphism $\phi$ : $X\longrightarrow Y$ such that for all $T\in\Gamma,$ there exists $S\in\Gamma_1$ satisfying $S\circ\phi=\phi\circ T$. Assume that $\Gamma$ is strictly transitive in $X$.
Let $x$, $y\in Y$. There exist $a,$ $b\in X$ such that $\phi(a)=x$ and $\phi(b)=y$. Since $\Gamma$ is strictly transitive in $X$, there exist $\alpha\in\mathbb{C}$ and $T\in \Gamma$ such that $\alpha Ta=b.$ Let $S\in\Gamma_1$ such that $S\circ\phi=\phi\circ T$, this implies that $\alpha Sx=y$. Hence $\Gamma_1$
is strictly transitive in $Y$.

For the converse, we do the same proof with $\phi^{-1}$ the invertible operator of $\phi,$ and the proof is completed.
\end{proof}
 In the following theorem, the proof is also true for norm-density if $X$ is assumed to be a normed linear space. 
\begin{theorem}
Let X be a topological vector space. Then for each pair of nonzero linearly independent vectors $x$, $y\in X$ there exists
a SOT-dense set $\Gamma_{xy}\subset\mathcal{B}(X)$ that is not strictly transitive. Furthermore, $\Gamma\subset\mathcal{B}(X)$ is a dense nonstrictly
transitive set if and only if $\Gamma$ is a dense subset of $\Gamma_{xy}$ for some $x$, $y\in X.$
\end{theorem}
\begin{proof}
Fix nonzero linearly independent vectors $x,$ $y\in X$ and put 
$$\Gamma_{xy}=\{T\in\mathcal{B}(X) \mbox{ : }y\mbox{ and }Tx \mbox{ are linearly independent}\}.$$
It is clear that $\Gamma_{xy}$ is not strictly transitive. Let $\Omega$ be a nonempty
open set in $\mathcal{B}(X)$ and $S\in\Omega$. If $Sx$ and $y$ are linearly independent, then $S\in\Omega\cap\Gamma_{xy}$.
Otherwise, putting $S_n = S+\frac{1}{n}I$, we see that $S_k\in\Omega$ for some $k$, but $S_kx$ and $y$ are linearly independent. Hence, $\Omega\cap\Gamma_{xy}\neq\emptyset$ and the proof is completed.

We prove the second assertion of the theorem. Suppose that $\Gamma$ is a dense subset of $\mathcal{B}(X)$ that is not
strictly transitive. Then there exist nonzero vectors $x,$ $y\in X$ such that $Tx$ and $y$ are linearly independent for all $T\in \Gamma$ and hence $\Gamma\subset \Gamma_{xy}$.
To show that $\Gamma$ is dense in $\Gamma_{xy}$, assume that $\Omega_0$ is an open subset of $\Gamma_{xy}$. Thus, $\Omega_0= \Gamma_{xy}\cap\Omega$ for some open
set $\Omega$ in $\mathcal{B}(X)$. Then $\Gamma\cap \Omega_0= \Gamma\cap \Omega\neq\emptyset$.

For the converse, let $\Gamma$ be a dense subset of $\Gamma_{xy}$ for some $x,$ $y\in X$. Then $\Gamma$ is not strictly transitive. Also, since $\Gamma_{xy}$ is a dense open subset of $\mathcal{B}(X)$, we conclude that $\Gamma$ is also dense in $\mathcal{B}(X)$. Indeed,
if $\Omega$ is any open set in $\mathcal{B}(X)$ then $\Omega\cap\Gamma_{xy}\neq\emptyset$ since $\Gamma_{xy}$ is dense in $\mathcal{B}(X)$. On the other hand, $\Omega\cap\Gamma_{xy}$
is open in $\Gamma_{xy}$ and so it must intersect $\Gamma$ since $\Gamma$ is dense in $\Gamma_{xy}$. Thus,  
$\Omega\cap\Gamma\neq\emptyset$ and so $\Gamma$ is dense in $\mathcal{B}(X)$.
\end{proof}
\begin{corollary}
Let X be a topological vector space and $\Gamma$ be a dense subset of $\mathcal{B}(X)$. Then there is a subset $\Gamma_1$
of $\Gamma$ such that $\overline{\Gamma_1}= \mathcal{B}(X)$ and $\Gamma_1$ is not strictly transitive.
\end{corollary}
\begin{proof}
For nonzero linearly independent vectors $x,$ $y$ put $\Gamma_1=\Gamma\cap\Gamma_{xy}$.
\end{proof}
\begin{definition}
A set $\Gamma\subset\mathcal{B}(X)$ is said to be supertransitive if $SC(\Gamma)=X\setminus\{0\}.$
An operator $T\in \mathcal{B}(X)$ is supertransitive if 
$$\Gamma=\{T^n\mbox{ : }n\geq0\} $$
is supertransitive.
\end{definition}

It is clear that a supertransitive set is supercyclic. Moreover, the next proposition shows that supertransitivity of sets of operators implies supercyclic transitivity.
\begin{proposition}
Let $\Gamma$ be a subset of $\mathcal{B}(X)$. If $\Gamma$ is supertransitive, then $\Gamma$ is supercyclic transitive. 
\end{proposition}
\begin{proof}
Let $U$ and $V$ be two nonempty open subsets of $X$. There exists $x\in X\setminus\{0\}$ such that $x\in U$. Since $\Gamma$ is supertransitive, there exists $\alpha\in\mathbb{C}$ and $T\in\Gamma$ such that $\alpha Tx\in V$. This implies that $\alpha T(U)\cap V\neq\emptyset.$ Hence, $\Gamma$ is supercyclic transitive.
\end{proof}
\begin{proposition}
Let $X$ and $Y$ be topological vector spaces and let $\Gamma\subset \mathcal{B}(X)$ be similar to $\Gamma_1\subset \mathcal{B}(Y)$. Then, $\Gamma$ is supertransitive on $X$ if and only if $\Gamma_1$ is supertransitive on $Y$.
\end{proposition}
\begin{proof}
Since $\Gamma$ and $\Gamma_1$ are similar, there exists a homeomorphism $\phi$ : $X\longrightarrow Y$ such that for all $T\in\Gamma,$ there exists $S\in\Gamma_1$ satisfying $S\circ\phi=\phi\circ T$. 
If $\Gamma$ is a supertransitive on $X$, then by Proposition \ref{14}, $\phi(SC(\Gamma))\subset SC(\Gamma_1)$. Since $\phi$ is homeomorphism, the result holds.

For the converse, we do the same proof by using $\phi^{-1}$ the invertible operator of $\phi$, and the proof is completed.
\end{proof}

Assume that $X$ is a topological vector space and $\Gamma\subset\mathcal{B}(X)$. The following result shows that the SOT-closure of $\Gamma$ is not large enough to have more supercyclic vectors than $\Gamma$.
\begin{proposition}
Let $X$ be a topological vector space and $\Gamma\subset\mathcal{B}(X)$. If $\overline{\Gamma}$ stands for the SOT-closure of $\Gamma$ then 
$$SC(\Gamma)=SC(\overline{\Gamma}).$$
\end{proposition}
\begin{proof}
We only need to prove that $SC(\overline{\Gamma})\subset SC(\Gamma)$. Fix $x\in SC(\overline{\Gamma})$ and let $U$ be an arbitrary open subset of $X$. Then there is some $\alpha\in\mathbb{C}$ and $T\in\overline{\Gamma}$ such that $\alpha Tx \in U$. The set $\Omega = \{S\in\mathcal{B}(X) \mbox{ : } \alpha Sx \in U\}$ is a SOT-neighborhood
of T and so it must intersect $\Gamma$. Therefore, there is some $S\in \Gamma$ such that $\alpha Sx \in U$ and this shows that
$x\in SC(\Gamma)$.
\end{proof}
\begin{corollary}
 Let X be a topological vector space and $\Gamma\subset\mathcal{B}(X)$. Then $\Gamma$ is supertransitive if and only if $\overline{\Gamma}$ is supertransitive.
\end{corollary}
In the next definition, we introduce the supercyclic criterion for a sets of operators.
\begin{definition}\label{cc}
Let $\Gamma$ be a subset of $\mathcal{B}(X)$. We say that $\Gamma$ satisfies the criterion of supercyclicity if there exist two dense subsets $X_0$ and $Y_0$ in $X$, a sequence $\{k\}$ of positives integers, a sequence $\{\alpha_k\}$ of nonzero complex numbers, a sequence $\{T_k\}$ of $\Gamma$ and a sequence of maps $S_k$ : $Y_0\longrightarrow X$ such that$:$
\begin{itemize}
\item[$(i)$] $\alpha_k T_kx\longrightarrow 0$ for all $x\in X_0$;
\item[$(ii)$] $\alpha_k^{-1} S_kx\longrightarrow 0$ for all $y\in Y_0$;
\item[$(iii)$] $T_kS_ky\longrightarrow y$ for all $y\in Y_0$.
\end{itemize}
\end{definition}
\begin{remark}
Let $T\in\mathcal{B}(X)$. Then $T$ satisfies the criterion of supercyclicity as an operator if and only if the set
$$\Gamma=\{T^n \mbox{ : }n\in\geq0\}$$
satisfies the criterion of supercyclicity as a set of operators. 
\end{remark}
\begin{theorem}\label{11}
Let $X$ be a second countable Baire topological vector space and $\Gamma$ a subset of $\mathcal{B}(X).$ If $\Gamma$ satisfies the criterion of supercyclicity, then $SC(\Gamma)$ is a dense subset of $X.$ As consequence$;$ $\Gamma$ is supercyclic.
\end{theorem}
\begin{proof}
Let $U$ and $V$ be two nonempty open subsets of $X$. Since $X_0$ and $Y_0$ are dense in $X$, there exist $x_0$ and $y_0$ in $X$ such that 
$$ x_0\in X_0\cap U\hspace{0.3cm}\mbox{ and }\hspace{0.3cm} y_0\in Y_0\cap V.$$
For all $k\geq1$, let $z_k=x_0+\alpha_k^{-1} S_ky $. By Definition \ref{cc}, we have 
$\alpha_k^{-1} S_k y\longrightarrow 0$ which implies that $z_k\longrightarrow x_0$. Since $x_0\in U$
and $U$ is open, there exists $N_1\in\mathbb{N}$ such that $z_k\in U$, for all $k\geq N_1$. On the other hand, $\alpha_k T_k z_k=\alpha_k T_k x_0+T_k (S_k y_0)\longrightarrow y_0$. Since $y_0\in V$ and $V$ is open, there exists $N_2\in\mathbb{N}$ such that $\alpha_k T_k z_k\in V$, for all $k\geq N_2$. Let $N=$max$\{N_1,N_2\}$, then $z_k\in U$ and $\alpha_k T_k z_k\in V$, for all $k\geq N$. From this, we deduce that 
$$\alpha_k T_k (U)\cap V\neq \emptyset.$$
for all $k\geq N$.
Hence, $\Gamma$ is a supercyclic transitive set. By Theorem \ref{t1}, we deduce that $SC(\Gamma)$ is a dense subset of $X$.
\end{proof}
\begin{remark}
If $X$ is a complex normed space and $\Gamma$ is a subset of $\mathcal{B}(X)$ which satisfies the criterion of supercyclicity, then by Theorem \ref{11}, $\Gamma$ is a supercyclic transitive set. Thus, $\Gamma$ satisfies the conditions $(ii)$ and $(iii)$ of Theorem \ref{tt}.
\end{remark}
\section{Supercyclic $C$-Regularized Groups}

In this section, we study the particular case where $\Gamma$ stands for a $C$-regularized group. 
Recall \cite{CKM}, that an entire $C$-regularized group is an operator family $(S(z))_{z\in\mathbb{C}}$ on $\mathcal{B}(X)$ that satisfies:
\begin{itemize}
\item[$(1)$] $S(0)=C;$
\item[$(2)$] $S(z+w)C = S(z)S(w)$ for every $z,$ $w\in\mathbb{C}$,
\item[$(3)$] The mapping $z \mapsto S(z)x$, with $z\in\mathbb{C}$, is entire for every $x \in X$.
\end{itemize}
\begin{lemma}\label{lem}
Let $(S(z))_{z\in\mathbb{C}}$ be a supercyclic $C$-regularized group. If $C$ has dense range, then $Cx\in SC((S(z))_{z\in\mathbb{C}})$, for all $x\in SC((S(z))_{z\in\mathbb{C}}).$
\end{lemma}
\begin{proof}
This, since $C\in \{(S(z))_{z\in\mathbb{C}}\}^{'}$.
\end{proof}

By Theorem \ref{t1}, every supercyclic transitive $C$-regularized group is supercyclic. In the following we prove that the converse is holds. 
\begin{theorem}
Let $(S(z)_{z\in\mathbb{C}})$ be a $C$-regularized group such that $C$ has dense range. If $(S(z)_{z\in\mathbb{C}})$  is supercyclic, then $(S(z)_{z\in\mathbb{C}})$ is supercyclic transitive.
\end{theorem}
\begin{proof}
Let $x\in SC((S(z))_{z\in\mathbb{C}})$. If $U$ and $V$ are two nonempty open subsets of $X$, then there exist $\alpha,$ $\beta\in\mathbb{C}$ and $z_1$, $z_2\in\mathbb{C}$  such that
\begin{equation}\label{e11}
\alpha S(z_1)x\in C^{-1}(U) \hspace{0.6cm}\mbox{ and }\hspace{0.6cm}\beta S(z_2)x\in V. 
\end{equation}
Let $z_3=z_1-z_2$. By $\ref{e11}$, we have
$$\alpha S(z_3)(S(z_2)x)\in U\hspace{0.3cm}\mbox{ and }\hspace{0.3cm}\beta S(z_3)(S(z_2)x)\in  S(z_3)(V),$$
which implies that $U\cap\frac{\beta}{\alpha} S(z_3)(V)\neq \emptyset$. 
Hence, $(S(z))_{z\in\mathbb{C}}$ is a supercyclic transitive $C$-regularized group.
\end{proof}
\begin{theorem}
Let $(S(z)_{z\in\mathbb{C}})$ be a supercyclic $C$-regularized group on a Banach infinite dimensional space $X$. Assume that $C$ is of dense range. If $x \in X$ is a supercyclic vector of $(S(z)_{z\in\mathbb{C}})$, then the following assertions hold:
\begin{itemize}\label{l1}
\item[$(1)$] $S(z) x\neq 0$ for all $z\in\mathbb{C}$;
\item[$(2)$] The set $\{\alpha S(z) x \mbox{ : }\alpha\mbox{, }z\in\mathbb{C}\mbox{, }\vert z\vert>\vert \omega_0\vert\}$ is dense in $X$ for all $\omega_0\in\mathbb{C}$.
\end{itemize}
\end{theorem}

\begin{proof}
$(1)$ If $z_1\in\mathbb{C}$ is such that $S(z_1)x=0$, then $S(z_1)(Cx)=0$. Let $z_2\in \mathbb{C}$. Then
$$ S(z_2)(Cx)=S(z_2-z_1+z_1)(Cx)=S(z_2-z_1)(S(z_1)x)=0. $$
This is a contradiction, since by Lemma \ref{lem} $Cx\in SC(S(z))_{z\in\mathbb{C}}).$\\
$(2)$ Suppose that there exists $\omega_0\in\mathbb{C}$ such that $A:=\{\alpha S(z)x \mbox{ : }\alpha\mbox{, }z\in\mathbb{C}\mbox{, }\vert z \vert>\vert \omega_0 \vert\}$ is not dense in X. Hence there exists a bounded open set $U$ such that $U \cap \overline{A}= \emptyset$.
Therefore we have
$$ U\subset \overline{\{\alpha S(z)x \mbox{ : }\alpha\mbox{, }z\in\mathbb{C}\mbox{, }\vert z \vert\leq\vert \omega_0 \vert\}} $$
by using the relation 
$$ X=\overline{\{\alpha S(z)x \mbox{ : }\alpha\mbox{, }z\in\mathbb{C}\}}=\overline{\{\alpha S(z)x \mbox{ : }\alpha\mbox{, }z\in\mathbb{C}\mbox{, }\vert z \vert>\vert \omega_0 \vert\}}\cup \overline{\{\alpha S(z)x \mbox{ : }\alpha\mbox{, }z\in\mathbb{C}\mbox{, }\vert z \vert\leq\vert \omega_0 \vert\}}. $$
Since $S(z)x$ is continuous with $z$ and $S(z)x \neq0$  holds for all $z\in\mathbb{C}$ by $(1)$, there exist $m_1$, $m_2 >0$ such that $0 < m_1 \leq\Vert S(z)x\Vert < m_2$ for $z\in\mathbb{C}$ with $\vert z\vert\leq\vert \omega_0\vert$.
There exists $M > 0$ such that $\Vert y\Vert \leq M$ for any $y \in U$ because $U$ is bounded. So we have $$U \subset \overline{\left\lbrace \alpha S(z) x \mbox{ : }\vert z\vert\leq\vert \omega_0\vert \mbox{, }\vert \alpha \vert\leq \frac{M}{m_1} \right\rbrace  },$$
 which means that $\overline{U}$ is compact. Hence $X$ is finite dimensional, which contradicts that $X$ is infinite dimensional.
\end{proof}
\section{Supercyclic Strongly Continuous Semigroups}
 
Recall that a one-parameter family $(T_t)_{t\geq0}$ of operators on $X$ is called a strongly continuous semigroup of operators if the following three conditions are satisfied$:$
\begin{itemize}
\item[$(i)$] $T_0=I$ the identity operator on $X$;
\item[$(ii)$] $T_{t+s}=T_{t}T_{s}$ for all $t,$ $s\geq0$;
\item[$(iii)$] $\lim_{t\rightarrow s}T_{t}x=T_{s}x$ for all $x\in X$ and $t\geq 0$.
\end{itemize}
One also refers to it as a $C_0$-semigroup.

The linear operator defined in 
$$ D(A)=  \left\lbrace x\in X\mbox{ : }\lim_{t\downarrow0}\frac{T_tx-x}{t}\mbox{ exists }\right\rbrace   $$
by
$$ Ax=\lim_{t\downarrow0}\frac{T_tx-x}{t}=\frac{d^{+}T_tx}{dt}\vert_{t=0}\mbox{ for }x\in D(A) $$
is the infinitesimal generator of the $C_0$-semigroup $(T_t)_{t\geq0}$ and $D(A)$ is the domain of $A$.
For more informations about the theory of $C_0$ semigroups we refer to books by Engel and Nagel \cite{Nagel,Engel}, by Pazy \cite{Pazy} and by Dunford  and Schwartz \cite{Dunford Schwartz}.

As a consequence of Theorem \ref{so}, we may prove that the supercyclicity and the supercyclic transitivity are two notions equivalent in the case of strongly continuous semigroups of operators.
\begin{theorem}\label{t4}
Let $\Gamma=(T_t)_{t\geq0}$ be a strongly continuous semigroup of operators. Then, the following assertions are equivalent$:$
\begin{itemize}
\item[$(i)$]$\Gamma$ is supercyclic;
\item[$(ii)$]$\Gamma$ is supercyclic transitive.
\end{itemize}
\end{theorem}
\begin{proof}
 By remarking that if $t_1>t_2\geq0$, then there exists $t=t_1-t_2$ such that $T_{t_1}=T_t T_{t_2}$ and applying Theorem \ref{so}.
\end{proof}
\begin{definition}\cite{Pazy}
Let $(T_t)_{t\geq0}$ be a strongly continuous semigroup of operators on $X$. Given another topological vector space $Y$ and an isomorphism $\phi$ from $Y$ onto $X$, the strongly continuous semigroup of operators $(S_t)_{t\geq0}$ on $Y$, defining by
$$ S_t=\phi^{-1} T_t \phi $$
for $t\geq0$, is said to be similar to $(T_t)_{t\geq0}$.
\end{definition}
\begin{proposition}
Let $(T_t)_{t\geq0}$ be a supercyclic strongly continuous semigroup of operators on $X$. If $(S_t)_{t\geq0}$ is a strongly continuous semigroup of operators on $Y$ similar to $(T_t)_{t\geq0}$, then $(S_t)_{t\geq0}$ is supercyclic on $Y$. Moreover, 
$$SC((S_t)_{t\geq0})=\phi(SC((T_t)_{t\geq0})).$$
\end{proposition}
\begin{proof}
Direct consequence of Proposition \ref{14}.
\end{proof}
\begin{definition}\cite{Pazy}
Let $(T_t)_{t\geq0}$ be a strongly continuous semigroup of operators. For any numbers $\mu\in\mathbb{C}$ and $\alpha > 0$, we define the rescaled strongly continuous semigroup of operators $(S_t)_{t\geq0 }$ by
$$ S_t= e^{\mu t}T_{(\alpha t)}$$
for $t\geq0$.
\end{definition}
\begin{proposition}
Let $(T_t)_{t\geq0}$ be a supercyclic strongly continuous semigroup of operators. For any numbers $\mu\geq0$, the rescaled strongly continuous semigroup of operators $(S_t)_{t\geq0 }=(e^{\mu t}T_{ t})_{t\geq0}$ is supercyclic. 
\end{proposition}
\begin{proof}
It suffices to take $\alpha_t=e^{\mu t}$ for all $t\geq0$ in Proposition \ref{p3}.
\end{proof}


\begin{thebibliography}{00}


\bibitem {AO} M.  Amouch, O. Benchiheb.
\it{ On cyclic sets of operators}, DOI: 10.1007/s12215-018-0368-4
\bibitem {AKH} M.  Ansari, B. Khani-robati, K. Hedayatian.
\it{On the density and transitivity of sets of operators}, Turk. J. Math., 42  (2018), 181-189.
\bibitem{Bayart Matheron} F. Bayart, E. Matheron.
\it{Dynamics of linear operators}, In: Cambridge Tracts in Mathematics, vol. 179. Cambridge University Press, Cambridge (2009).
\bibitem{BBP} T. Bermudez, A. Boinlla, A. Peris.
{\it On Hypercyclicity and Supercyclicity Criterias.} Bull. Austral. Math. Soc., 70 (2004), pp 45-54.
\bibitem{Bes} J. P. B$\grave{e}$s.
{\it Three problem's on hypercyclic operators.} Ph.D. thesis, Bowling Green State University, Bowling Green, Ohio, 1998.
\bibitem{CKM} J. A. Conejero, M. Kosti$\acute{c}$, P. J. Miana, M. Murillo-Arcila.
{\it  Distributionally chaotic families of operators on Frechet spaces.}  Comm. Pure Appl. Anal., to appear.
\bibitem{Conway} J.B. Conway. {\it A course in Functional Analysis.} Springer Graduate Texts in Math Series, 1985
\bibitem{Dunford Schwartz} L. Dunford and J. T. Schwartz.
{\it Linear Operators.} Part I, Interscience, New York, 1958.
\bibitem{Nagel} K.-J. Engel and R. Nagel.
{\it A short course on operator semigroups.} Springer, New
York, 2006.
\bibitem{Engel} K.-J. Engel and R. Nagel.
{\it One-parameter semigroups for linear evolution equations.} Springer, New York-Berlin, 2000.
\bibitem{Feldman1} N. S. Feldman.
{\it  Hypercyclic and supercyclic for invertible bilateral weighted shifts.} J. Math. Analysis and Appl, 1, (2001), 67-74.
\bibitem{Gethner Shapiro} R. M. Gethner and J. H. Shapiro.
{\it Universal vectors for operators on spaces of holomorphic functions.} Proc. Amer. Math. Soc., 100 (2): 281-288, 1987.
\bibitem{GEU} K.-G. Grosse-Erdmann.
{\it Universal families and hypercyclic operators.} Bull. Amer. Math. Soc 36: 345-381, 1999.
\bibitem{Erdmann Peris} K.-G. Grosse-Erdmann and A. Peris Manguillot.
{\it Linear chaos.} Universitext, Springer, London, 2011.
\bibitem{HL} D. A. Herrero.
{\it Limits of hypercyclic and supercyclic operators.} Universitext, Springer, London, 2011. /. Fund. Anal. 99(1991), 179-
\bibitem{HW} H. M. Hilden, L. J. Wallen. {\it Some cyclic and non-cyclic vectors of certain operators.} Indiana Univ. Math. J., 23 :557-565, 1973/74. 
\bibitem{LSM} F. Leon- Saavedra, V. M$\ddot{u}$ller.
{\it Rotations of hypercyclic and supercylic operators.} Integral Equations Operator Theory. 50 (2004), 385-391. 
\bibitem{MS} A. Montes-Rodriguez,  H. N. Salas.
\it{ Supercyclic subspaces$:$ spectral theory and weighted shifts}, Adv. Math. 163 (2001), no. 1, 74-134.
\bibitem{Pazy} A. Pazy.
{\it Semigroups of Linear Operators and Applications to Partial Differential Equations.} Springer-Verlag, New York, 1983. 
\bibitem{Rolewicz} S. Rolewicz.
{\it  On orbits of elements.} Studia Math., 32: 17-22, 1969. 
\bibitem{Salas2}  H.N. Salas,
{\it Supercyclicity and weighted shifts.}  Studia Math., 135(1) :55-74, 1999. 
\bibitem{Willered} S. Willered, 
\it{General Topology}. Addison-Wesley, 1968.

\end{thebibliography}
\end{document}